\documentclass{amsart}

\usepackage{amssymb}
\usepackage{url} 

\newtheorem{theorem}{Theorem}[section]
\newtheorem{corollary}[theorem]{Corollary}
\newtheorem{proposition}[theorem]{Proposition}
\newtheorem{definition}[theorem]{Definition}
\newtheorem{remark}[theorem]{Remark}
\newtheorem{lemma}[theorem]{Lemma}

\newcommand{\Gm}{{\mathfrak m}}
\newcommand{\Gin}{{\rm Gin}}
\newcommand{\Hilb}{{\rm H}}
\newcommand{\Gp}{{\mathfrak p}}
\newcommand{\Gq}{{\mathfrak q}}
\newcommand{\Ass}{{\rm Ass}}

\newcommand{\GB}{{\mathfrak B}}

\begin{document}

\title[Completely $\Gm$-full ideals and componentwise linear ideals]
{Completely $\Gm$-full ideals \\ and componentwise linear ideals}


\author{Tadahito Harima}

\address{
Department of Mathematics Education \\
Niigata University \\
Niigata 950-2181 \\
JAPAN
}

\email{harima@ed.niigata-u.ac.jp}

\author{Junzo Watanabe}
\address{Department of Mathematics, Tokai University, 
    Hiratsuka 259-1292, JAPAN
}

\email{watanabe.junzo@tokai-u.jp}

\thanks{
The first author was  supported by Grant-in-Aid for Scientific Research (C) (23540052). 
The second author was  supported by Grant-in-Aid for Scientific Research (C) (24540050). 
}




\begin{abstract}
We show that the class of completely $\Gm$-full ideals
coincides with the class of componentwise linear ideals 
in a polynomial ring over an infinite field.
\end{abstract}

\maketitle


\section{Introduction}

The notion of completely $\Gm$-full ideals in a local ring 
was introduced by the second author \cite{Watanabe2}, 
and the notion of componentwise linear ideals in a polynomial ring 
was introduced by Herzog and Hibi  \cite{HH}. 
These ideals are two important classes of ideals having various interesting properties. 
In \cite{HW} the authors proved that 
these notion are equivalent in the class of graded ideals 
provided that their generic initial ideals with respect to 
the graded reverse lexicographic order are stable, 
and further conjectured that 
these notions are equivalent without adding the assumption on generic initial ideals. 
The purpose of this paper is to prove that the conjecture is true. 
The following is the main theorem.

\begin{theorem}\label{main theorem} 
Let $K$ be an infinite field of any characteristic 
and $I$ a graded ideal of the polynomial ring $R=K[x_1,\ldots,x_n]$. 
Then 
$I$ is a completely $\Gm$-full ideal if and only if it is a componentwise linear ideal. 
\end{theorem}

The ``if'' part was proved in Proposition 18 of \cite{HW}. 
So we will show the ``only if'' part. 
For the proof of the ``only if'' part, 
we use the characterization theorem for componentwise linear ideals 
by Nagel and R\"{o}mer \cite{NR}. 
Their result says that 
the following conditions are equivalent for a graded ideal $I$ of $R=K[x_1,\ldots,x_n]$.  
\begin{enumerate}
\item[(i)] 
$I$ is a componentwise linear ideal. 
\item[(ii)] 
The generic initial ideal $\Gin (I)$ of $I$ is stable and $\mu(I)=\mu(\Gin(I))$, 
where $\mu$ denotes the minimal number of generators of an ideal. 
\end{enumerate}

In Section 5 
we prove that if $I$ is a completely $\Gm$-full ideal 
then $\Gin(I)$ is stable and $\mu(I)=\mu(\Gin(I))$. 
In Section 2 
we summarize basic notation and definitions. 
In Section 3 
we introduce the notion of $t$-sequences for a graded ideal, 
and in Section 4 we give a characterization of completely $\Gm$-full ideals 
in terms of the $t$-sequences. 
It plays a key role in the proof of Theorem~\ref{main theorem}. 
In section 6 we show that 
a theorem of Nagel-R\"{o}mer  (\cite{NR}) is 
an immediate consequence of Thoerem~\ref{main theorem}.


\section{Notation and definitions}

Throughout this paper, 
we let $K$ be an infinite field of any characteristic, 
$R=K[x_1,\ldots,x_n]$ the polynomial ring in $n$ variables over $K$ with the standard grading, 
and $\Gm=(x_1,\ldots,x_n)$ the graded maximal ideal.  
Let $\Gin(I)$ denote the generic initial ideal of an ideal $I$ of $R$ 
with respect to  the graded reverse lexicographic order 
induced by $x_1>\cdots>x_n$, 
and $\Hilb(M,j)=\dim_KM_j$ 
the Hilbert function of a graded module $M=\oplus_{j\geq 0}M_j$ over $R$. 
Let $l$ and $\mu$ be the length and the minimal number of generators 
of a graded ideal $I$ in $R$, respectively, 
hence $\mu (I) = l (I/\Gm I)$. 
The type of a graded ideal $I$ is the length of $(I:\Gm)/I$ as an ideal of $R/I$. 
It is equal to the last free rank in the minimal free resolution of $R/I$. 
\medskip

The definition  of {\em $\Gm$-full ideal} is due to Rees. 
We adapt the definition to  graded ideals as follows.  

\begin{definition}[\cite{Watanabe}]
{\rm 
A graded ideal $I$ of $R=K[x_1,\ldots,x_n]$ is said to be {\em $\Gm$-full} 
if there exists an element $z$ of $R$ such that $\Gm I:z = I$. 
}
\end{definition}

\begin{remark}
{\rm 
Suppose that $I$ is  an $\Gm$-full ideal of $R$.  
Then the equality $\Gm I:z = I$ holds 
for a general linear form $z$ of $R$ (\cite[Remark 2 (i)]{Watanabe}). 
}
\end{remark}

We adapt the original definition of completely $\Gm$-full  ideals 
(defined in \cite{Watanabe2}) to the graded ideals as follows.

\begin{definition}[\cite{Watanabe2}] \label{completely m-full}
{\rm 
Let $I$ be a graded ideal of $R=K[x_1,\ldots,x_n]$. 
We define the {\em completely $\Gm$-full ideals} recursively as follows. 
\begin{enumerate}
\item[(1)] 
If $n=0$ (i.e., if $R$ is a field), then the zero ideal is completely $\Gm$-full.  
\item[(2)] 
If $n>0$, 
then $I$ is completely $\Gm$-full 
if $\Gm I:z=I$ and $(I+zR)/zR$ is completely $\Gm$-full as an ideal of $R/zR$, 
where $z$ is a general linear form in $R$. 
(The definition makes sense by induction on $n$.) 
\end{enumerate}
}
\end{definition}

\begin{remark}
{\rm 
For a completely $\Gm$-full ideal $I$ of $R=K[x_1,\ldots,x_n]$, 
there exist $n$ general linear forms $z_n,z_{n-1},\ldots,z_1$ in $R$ 
satisfying the following conditions: 
\begin{enumerate}
\item[(i)] 
$\Gm I:z_n = I$, i.e., $I$ is $\Gm$-full. 
\item[(ii)] 
$\overline{\Gm}\overline{I}:\overline{z_{n-i+1}} = \overline{I}$ 
in $\overline{R}=R/(I,z_n,\ldots,z_{n-i+2})$ for all $i=2,3,\ldots,n$, 
where $\overline{\ast}$ denotes the reduction mod $(I,z_n,\ldots,z_{n-i+2})$. 
\end{enumerate} 
In this case we say that 
$(I;z_n,z_{n-1},\ldots,z_1)$ has the complete $\Gm$-full property. 
}
\end{remark}

\begin{definition}[\cite{HH}] \label{componentwise linear}
{\rm 
If $I$ is a graded ideal of $R=K[x_1,\ldots,x_n]$, 
then we write $I_{<j>}$ for the ideal generated by 
all homogeneous polynomials of degree $j$ belong to $I$. 
We say that a graded ideal $I$ of $R$ is {\em componentwise linear} 
if $I_{<j>}$ has a linear resolution for all $j$. 
}
\end{definition}

\begin{definition}
{\rm 
A monomial ideal $I$ of $R=K[x_1,\ldots,x_n]$ is said to be {\em stable} 
if $I$ satisfies the following condition: 
for each monomial $u\in I$, 
the monomial $x_iu/x_{m(u)}$ belongs to $I$ for every $i<m(u)$, 
where $m(u)$ is the largest index $j$ such that $x_j$ divides $u$. 
}
\end{definition}

It is known that 
stable ideals are completely $\Gm$-full 
(\cite[Section 4]{Watanabe3} and \cite[Example 17]{HW}), 
and also componentwise linear (\cite[Example 1.1]{HH}).


\section{The $t$-sequence of a graded ideal }

The second author \cite{Watanabe3} defined 
the $t$-sequence for a completely $\Gm$-full ideal. 
In this section we extend the notion of $t$-sequences to graded ideals in general. 
The following is a revised version of Theorem C in \cite{Watanabe}.

\begin{theorem}\label{thm: t-sequence}
Let $X_1,\ldots,X_n$ be indeterminates over $R=K[x_1,\ldots,x_n]$, 
let $R^\prime$ denote $S=R[X_1,\ldots,X_n]$ localized at $\Gm R[X_1,\ldots,X_n]$ 
and set $Y=X_1x_1+\cdots+X_nx_n$ in $R^\prime$. 
Let $I$ be a graded ideal of $R$. 
Then we have the following. 
\begin{enumerate}
\item[(1)] 
$l((IR^\prime:_{R^\prime}Y)/IR^\prime)$ is finite. 
\item[(2)] 
$l((IR^\prime:_{R^\prime}Y)/IR^\prime) \leq l((I:_{R}y)/I)$ for all linear forms $y$ in $R$. 
\item[(3)] 
$l((IR^\prime:_{R^\prime}Y)/IR^\prime) = l((I:_{R}y)/I)$ for a general linear form $y$ in $R$. 
\end{enumerate}
\end{theorem}

To prove this theorem we prepare a lemma. 

\begin{lemma}\label{finiteness of colon}
Let $I$ be a graded ideal of $R=K[x_1,\ldots,x_n]$. 
Then $l((I:y)/I)$ is finite 
for general linear forms $y$ of $R$. 
\end{lemma}

\begin{proof} 
Let $\Ass (I)$ be the set of associated prime ideals of $I$.
If $\Ass (I)=\{\Gm\}$, 
then it is obvious that 
$l((I:y)/I)$ is finite for all linear forms $y$ in $R$ 
since $I$ is $\Gm$-primary. 
If $\Gm \not \in \Ass (I)$ 
then $I:y=I$ for 
a general linear form $y$ in $R$, 
because $y$  is a non-zero divisor for $R/I$ if y is general enough. 
Hence $l((I:y)/I)=0$ in this case. 
So we assume that $I$ is not $\Gm$-primary and $\Gm \in \Ass(I)$. 
Let $I=\cap_{i=1}^{u} \Gq_i$ be a minimal primary decomposition of $I$, 
where $\sqrt{\Gq_1} = \Gm$. 
Let  $y$ be a linear form of $R$ 
such that $y \not \in \Gp$ for all  $\Gp \in \Ass(I) \backslash \{\Gm\}$.  
It suffices to show that $l((I:y)/I)$ is finite. 
We have  $I:y=\cap_{i=1}^{u} (\Gq_i:y)$. 
Since  $\sqrt{\Gq_1:y} = \Gm$ and $\Gq_i:y = \Gq _i$ for $ i > 1$,  
one sees that 
 $(I:y)/I$ is annihilated by a power of $\Gm$. 
This implies that $l((I:y)/I)$ is finite. 
\end{proof}

\begin{proof}[Proof of Theorem~\ref{thm: t-sequence}.] 
By the first part of Theorem A in \cite{Watanabe}, 
we have the inequalities 
\begin{equation}\label{EQ:3.1}
l(R^\prime / (I+\Gm^{s+1})R^\prime + YR^\prime)  
\leq l(R / (I+\Gm^{s+1}) + yR)  
\end{equation}
for all $s\geq 0$. 
That is, the inequalities
\[ 
\sum_{j=0}^s \Hilb(R^\prime / (IR^\prime + YR^\prime) ,j) 
\leq \sum_{j=0}^s \Hilb(R / (I + yR),j)   
\] 
hold for all $s\geq 0$. 
From the exact sequence 
\[
0 \rightarrow (I:y)/I \rightarrow R/I \stackrel{\times y}{\rightarrow} 
R/I \rightarrow R/(I+yR) \rightarrow 0, 
\] 
it follows that 
\[
\Hilb((I:y)/I,j)=\Hilb(R/I,j)-\Hilb(R/I,j+1)+\Hilb(R/I+yR,j+1)
\] 
for all $j\geq 0$. 
Similarly it follows that 
\[
\Hilb((IR^\prime:Y)/IR^\prime,j)
=\Hilb(R^\prime/IR^\prime,j)-\Hilb(R^\prime/IR^\prime,j+1)
+\Hilb(R^\prime/(IR^\prime+YR^\prime),j+1)
\] 
for all $j\geq 0$. 
Hence, 
since $R/I$ and $R^\prime/IR^\prime$ have the same Hilbert function, 
we obtain the inequalities 
\[
\sum_{j=0}^s \Hilb((IR^\prime:Y) / IR^\prime , j) 
\leq \sum_{j=0}^s \Hilb((I:y) / I, j)   
\] 
for all $s\geq 0$. 
Furthermore 
it follows by Lemma~\ref{finiteness of colon} that, 
for a general linear form $y$ of $R$, 
the equalities  
\[
l((I:y)/I) = \sum_{j=0}^s \Hilb((I:y) / I, j)   
\] 
hold for all $s >> 0$. 
Therefore we have 
\[
\sum_{j=0}^s \Hilb((IR^\prime:Y) / IR^\prime , j) 
\leq l((I:y)/I) 
\] 
for all $s >> 0$. 
Thus the assertions (1) and (2) are easily verified. 
The assertion (3) is also easy, 
since the equality in (\ref{EQ:3.1}) holds for a general linear form $y$ of $R$ 
by the second part of Theorem A in \cite{Watanabe}. 
\end{proof}

\begin{definition}
{\rm 
With the same notation as Theorem~\ref{thm: t-sequence}, 
we define $t(I)$ for a graded ideal $I$ by 
\[
t(I)=l((IR^\prime:_{R^\prime}Y)/IR^\prime). 
\]
We call $t(I)$ the {\em $t$-value} of $I$. 
Note that the equality 
\[
t(I)={\rm Min} \{ l((I:y)/I) \mid \mbox{$y$ is a linear form of $R$} \}
\]
holds by Theorem~\ref{thm: t-sequence}. 
}
\end{definition}

\begin{definition}\label{def: t-sequence}
{\rm 
Let $\{z_1,\ldots,z_n\}$ be a set of generators of $\Gm$ consisting of general linear forms. 
Set 
\[ 
R^{(i)}=R/(z_{i+1},z_{i+2},\ldots,z_{n})R 
\]
for all $i=0,1,\ldots,n-1$, and $R^{(n)}=R$. 
Let $t_i=t_i(I)$ denote the $t$-value of $IR^{(i+1)}$ for all $i=0,1,\ldots,n-1$. 
Note that $t_{0}=1$. 
We call the sequence $t_0,t_1,\ldots,t_{n-1}$ the {\em $t$-sequence} of $I$. 
This is a generalization of the notion of $t$-sequences 
introduced by the second author in \cite{Watanabe3}. 
We will discuss it in Remark~\ref{extension of t-sequences} bellow. 
}
\end{definition}

\begin{remark}\label{t-seq of I and GinI} 
{\rm 
We show that the $t$-sequence of $I$ is independent of 
a choice of general generators of $\Gm$.
We use the same notation as Definition~\ref{def: t-sequence}. 
Let $\overline{z_i}$ be the image of $z_i$ in $R^{(i)}$. 
From the exact sequence 
\[
\begin{array}{rl}
0 \rightarrow (IR^{(i)}:\overline{z_{i}})/IR^{(i)}  & \rightarrow R^{(i)}/IR^{(i)} \\
 & \stackrel{\times \overline{z_{i}}}{\rightarrow} R^{(i)}/IR^{(i)} 
\rightarrow R^{(i)}/(IR^{(i)}+\overline{z_{i}}R^{(i)}) \rightarrow 0, 
\end{array}
\]
it follows that 
\[ 
\begin{array}{rcl}
\Hilb((IR^{(i)}:\overline{z_{i}})/IR^{(i)},j)  & = & \Hilb(R^{(i)}/(IR^{(i)}+\overline{z_{i}}R^{(i)}),j+1) \\ 
& & \hspace{0.5cm} -\Hilb(R^{(i)}/IR^{(i)},j+1)+\Hilb(R^{(i)}/IR^{(i)},j) 
\end{array}
\] 
for all $j$. 
Set $R_{(i)}=R/(x_{i+1},x_{i+2},\ldots,x_n)R$ for all $i=0,1,\ldots,n-1$, 
$R_{(n)}=R$ and $J=\Gin(I)$. 
Similarly we get 
\[ 
\begin{array}{rcl}
\Hilb((JR_{(i)}:\overline{x_{i}})/JR_{(i)},j)  & = & \Hilb(R_{(i)}/(JR_{(i)}+\overline{x_{i}}R_{(i)}),j+1) \\ 
& & \hspace{0.5cm} -\Hilb(R_{(i)}/JR_{(i)},j+1)+\Hilb(R_{(i)}/JR_{(i)},j) 
\end{array}
\] 
for all $j$, 
where $\overline{x_i}$ is the image of $x_i$ in $R_{(i)}$. 
Hence, since 
\[ 
\begin{array}{rl}
 & \Hilb(R^{(i)}/IR^{(i)},j) = \Hilb(R_{(i)}/JR_{(i)},j) \\
\mbox{ and } & \Hilb(R^{(i)}/(IR^{(i)}+\overline{z_{i}}R^{(i)}),j) 
= \Hilb(R_{(i)}/(JR_{(i)}+\overline{x_{i}}R_{(i)}),j) 
\end{array}
\] 
for general linear forms $z_1,\ldots,z_n$ by \cite[Lemma 1.2]{C}, 
we have 
\[ 
l((IR^{(i)}:\overline{z_{i}})/IR^{(i)})  = l((JR_{(i)}:\overline{x_{i}})/JR_{(i)})  
\]  
for all $i>0$. 
This implies that 
the $t$-sequence of $I$ is independent of 
a choice of general generators of $\Gm$. 
This also implies that 
the $t$-sequence of $I$ coincides with that of $\Gin(I)$. 
}
\end{remark}

\begin{remark}\label{extension of t-sequences} 
{\rm 
With the same notation as Definition~\ref{def: t-sequence}, 
suppose that $(I;z_n,z_{n-1},\ldots,z_1)$ has the complete $\Gm$-full property. 
Let $t_0,t_1,\ldots,t_{n-1}$ be the $t$-sequence of $I$. 
The definition of $t$-sequences given in \cite[p. 238]{Watanabe3} implies that 
\[ 
t_i = \mu (IR^{(i+1)}) - \mu(IR^{(i)}) 
\] 
for all $i=0,1,\ldots,n-1$. 
Here note that $\mu(IR^{(0)})=0$ 
because $R^{(0)}=K$ and $IR^{(0)}=0$. 
Hence, since $(\Gm I)R^{(i+1)}:\overline{z_{i+1}} = IR^{(i+1)}$, 
it follows from Lemma~\ref{eq of m-full} in the next section that 
\[ 
t_i = l((IR^{(i+1)}:\overline{z_{i+1}})/IR^{(i+1)}). 
\]    
This means that 
Definition~\ref{def: t-sequence} gives a generalization of the notion of $t$-sequences 
for completely $\Gm$-full ideals. 
}
\end{remark}


\section{A characterization of completely $\Gm$-full ideals}

The purpose of this section is to prove the following. 

\begin{theorem}\label{char of cm-full}
Let $I$ be a graded ideal of $R=K[x_1,\ldots,x_n]$ 
and $t_0,t_1,\ldots,t_{n-1}$ the $t$-sequence of $I$. 
Set $B(I)=t_0+t_1+\cdots+t_{n-1}$. 
Then the following conditions are equivalent. 
\begin{enumerate}
\item[(i)] 
$I$ is a completely $\Gm$-full ideal. 
\item[(ii)] 
$\mu(I)=B(I)$. 
\end{enumerate}
\end{theorem}

We need a few lemmas for the proof of this theorem.

\begin{lemma}\label{inequality of m-full}
Let $I$ be a graded ideal of $R$, 
$z$ a linear form of $R$ 
and $\overline{I}$ the image of $I$ in $R/zR$. 
Then 
\begin{enumerate}
\item[(1)] 
$l ((\Gm I:z)/\Gm I) = \mu (\overline{I}) + l ((I:z)/I)$, and 
\item[(2)] 
$\mu (I) \leq \mu (\overline{I}) + l ((I:z)/I) $. 
\end{enumerate}
\end{lemma}

\begin{proof}
From the exact sequence 
\[ 
0 \rightarrow (\Gm I:z)/\Gm I \rightarrow R/\Gm I \stackrel{\times z}{\rightarrow} 
R/\Gm I \rightarrow R/(\Gm I+zR) \rightarrow 0, 
\]  
it follows that 
\[ 
\begin{array}{rcl}
\Hilb(I/\Gm I,j) &\leq& \Hilb((\Gm I:z)/\Gm I,j) \\
 &=& \Hilb(R/\Gm I,j)-\Hilb(R/\Gm I,j+1)+\Hilb(R/(\Gm I+zR),j+1) \\
 &=& \Hilb(R/\Gm I,j)-\Hilb(R/\Gm I,j+1) \\
 & &  \hspace{2cm} + \Hilb(R/(I+zR),j+1)+\Hilb((I+zR)/(\Gm I+zR),j+1) 
\end{array}
\] 
for all $j$. 
Similarly, 
from the exact sequence 
\begin{equation}\label{EQ:2}
0 \rightarrow (I:z)/I \rightarrow R/I \stackrel{\times z}{\rightarrow} 
R/I \rightarrow R/(I+zR) \rightarrow 0,    
\end{equation}
it follows that 
\[   
\Hilb(R/(I+zR),j+1)=\Hilb((I:z)/I,j)-\Hilb(R/I,j)+\Hilb(R/I,j+1) 
\] 
for all $j$. 
Hence we see 
\[ 
\begin{array}{rcl}
\Hilb(I/\Gm I,j) &\leq& \Hilb((\Gm I:z)/\Gm I,j) \\
 &=& \Hilb(I/\Gm I,j)-\Hilb(I/\Gm I,j+1) \\ 
 & & \hspace{2cm} +\Hilb((I:z)/I,j)+\Hilb((I+zR)/(\Gm I+zR),j+1) 
\end{array}
\] 
for all $j$. 
Thus we have 
\[  
\begin{array}{rcl}
\mu(I)=l(I/\Gm I) &\leq& l((\Gm I:z)/\Gm I) \\
 &=& l((I+zR)/(\Gm I+zR)) + l((I:z)/I) \\ 
 &=& \mu(\overline{I}) + l((I:z)/I). 
\end{array}
\]   
\end{proof}

\begin{lemma}\label{eq of m-full}
We use the same notation as Lemma~\ref{inequality of m-full}. 
Then the following conditions are equivalent. 
\begin{enumerate}
\item[(i)] 
$\Gm I:z=I$. 
\item[(ii)] 
$\mu(I)=\mu(\overline{I})+l ((I:z)/I)$. 
\end{enumerate}
Furthermore, if we assume that $I$ is $\Gm$-primary, 
these conditions are also equivalent to the following {\rm (iii)}. 
\begin{enumerate}
\item[(iii)] 
$\mu(I)=\mu(\overline{I})+l (R/(I+zR))$. 
\end{enumerate}
\end{lemma}

\begin{proof}
(i) $\Leftrightarrow$ (ii) is immediate from Lemma~\ref{inequality of m-full}. 
 (ii) $\Leftrightarrow$ (iii) is obvious, 
because $l ((I:z)/I)=l(R/(I+zR))$ from the exact sequence (\ref{EQ:2}) above 
if $I$ is $\Gm$-primary. 
\end{proof}

\begin{lemma}\label{inequality of t-seq}
Let $I$ be a graded ideal of $R=K[x_1,\ldots,x_n]$ 
and $t_0,t_1,\ldots,t_{n-1}$ the $t$-sequence of $I$. 
Set $B(I)=t_0+t_1+\cdots+t_{n-1}$. 
Then 
\[
\mu(I)\leq B(I). 
\] 
\end{lemma}

\begin{proof}
We use induction on $n$. 
If $n=1$, 
the equalities $\mu(I)=B(I)=1$ hold. 
Let $n>1$. 
By Lemma~\ref{inequality of m-full} (2), 
it follows that 
$\mu(I) \leq \mu(\overline{I}) + l((I:z)/I)$ 
for a general linear form $z$ of $R$. 
Furthermore the inductive assumption implies that 
$\mu(\overline{I}) \leq B(\overline{I})$. 
Hence we have 
\[
\mu(I) \leq B(\overline{I}) + l ((I:z)/I) = t_0+t_1+\cdots+t_{n-1} = B(I), 
\] 
as $B(\overline{I})=t_0+\cdots+t_{n-2}$ and $l((I:z)/I)=t_{n-1}$. 
\end{proof}

\begin{proof}[Proof of Theorem~\ref{char of cm-full}.] 
(i) $\Rightarrow$ (ii) follows from Corollary 9 in \cite{Watanabe2}. 
(ii) $\Rightarrow$ (i): 
We use induction on $n$. 
Let $n=1$. 
Then the equalities $\mu(I)=B(I)=1$ hold 
and any ideal of $K[x_1]$ is completely $\Gm$-full. 
Let $n>1$. 
Note that $B(I)=B(\overline{I})+t_{n-1}$. 
Hence we have that $B(\overline{I})\leq\mu(\overline{I})$, 
since $\mu(I)=B(I)$ 
and $\mu(I)\leq\mu(\overline{I})+t_{n-1}$ by Lemma~\ref{inequality of m-full} (2).  

On the other hand, 
the inequality $\mu(\overline{I})\leq B(\overline{I})$ holds 
by Lemma~\ref{inequality of t-seq}, 
and hence the equality $B(\overline{I})=\mu(\overline{I})$ holds. 
Therefore $\overline{I}$ is completely $\Gm$-full by the inductive assumption. 
Next we show that $I$ is $\Gm$-full. 
This follows from Lemma~\ref{eq of m-full} and the equalities 
\[
\mu(I)=B(I)=B(\overline{I})+t_{n-1}=\mu(\overline{I})+l((I:z)/I)
\] 
for a general linear form $z$ of $R$. 
\end{proof}

\begin{corollary}\label{Cor:reduction}
Let $I$ be a graded ideal of $R=K[x_1,\ldots,x_n]$, 
$x$ a non-zero divisor mod $I$ of degree one 
and $\overline{I}$ the image of $I$ in $R/xR$. 
Then $I$ is a completely $\Gm$-full ideal in $R$ 
if and only if 
$\overline{I}$ is a completely $\Gm$-full ideal in $R/xR$. 
\end{corollary}

\begin{proof} 
Since $x$ is a non-zero divisor mod $I$ of degree one, 
it follows that $\mu(\overline{I})=\mu(I)$.  
Furthermore we have $B(\overline{I})=B(I)$ 
by $l((I:x)/I)=0$. 
Hence this follows from Theorem~\ref{char of cm-full}. 
\end{proof}


\section{Proof of Main Theorem~\ref{main theorem} }

The following is a remark on a minimal generating set of an $\Gm$-full ideal.

\begin{remark}\label{notation of mfull} 
{\rm 
Suppose that $I$ is  an $\Gm$-full ideal of $R=K[x_1,\ldots,x_n]$.  
Then the equality $\Gm I:z = I$ holds 
for a general linear form $z$ of $R$. 
Moreover it is easy to see that, for any $z \in R$, if  
$\Gm I :z = I$, then it implies that  $I:\Gm = I:z$. 
Let $y_1,\ldots,y_s$ be homogeneous elements in $I:\Gm$ 
such that $\{\overline{y_1},\ldots,\overline{y_s}\}$ is 
a minimal generating set of $(I:\Gm)/I$, 
where $\overline{y_i}$ is the image of $y_i$ in $R/I$. 
Then Proposition 2.2 in \cite{GH} implies that 
$\{zy_1,\ldots,zy_s \}$ is a part of a minimal generating set of $I$. 
}
\end{remark}

We will prove Theorem~\ref{main theorem} 
after a series of lemmas.

\begin{lemma}\label{gen of mfull} 
With the same notation as Remark~\ref{notation of mfull}, 
write a minimal generating set of $I$ as 
\[  
zy_1,\ldots,zy_s, w_1,\ldots,w_t. 
\] 
Let $\overline{w_i}$ be the image of $w_i$ in $R/zR$ 
and $\overline{I}$ the image of $I$ in $R/zR$.  Then we have: 
\begin{enumerate}
\item[$(1)$] 
$\{\overline{w_1},\ldots,\overline{w_t}\}$ is a minimal generating set of $\overline{I}$. 
\item[$(2)$]
$\mu(I)=\mu(\overline{I})+l((I:\Gm)/I)$. 
\end{enumerate}
\end{lemma}

\begin{proof} 
(1) Suppose that $w_1\in (w_2,\ldots,w_t,z)$. 
Then 
\begin{equation}\label{EQ:3}
w_1=f_2w_2+\cdots+f_tw_t+f_{t+1}z   
\end{equation}
for some $f_i\in R$. 
Since $f_{t+1}z=w_1-(f_2w_2+\cdots+f_tw_t)\in I$, 
we have 
\[
f_{t+1} \in I:z = I: \Gm=(y_1,\ldots,y_{s}, w_1,\ldots,w_t). 
\]
Therefore  
\[ 
f_{t+1}=g_1y_1+\cdots+g_{s}y_{s}+h_1w_1+\cdots+h_tw_t 
\]
for some $g_i, h_j\in R$, and hence 
\begin{equation}\label{EQ:4}
zh_1w_1 = zf_{t+1}-z(g_1y_1+\cdots+g_s y_s)-z(h_2w_2+\cdots+h_tz_t).   
\end{equation}
Thus, from the equalities (\ref{EQ:3}) and (\ref{EQ:4}) above, 
we obtain 
\[
w_1-zh_1w_1=(f_2+zh_2)w_2+\cdots+(f_t+zh_t)w_t+g_1zy_1+\cdots+g_{s}zy_{s},  
\] 
and $w_1-zh_1w_1 \in (zy_1,\ldots,zy_{s},w_2,\ldots,w_t)$. 
Hence $w_1 \in (zy_1,\ldots,zy_{s},w_2,\ldots,w_t)$, 
since $\deg(w_1)<\deg(zh_1w_1)$. 
This is a contradiction. 

(2) immediately follows from (1), 
since $t=\mu(\overline{I})$ and $s=l((I:\Gm)/I)$. 
\end{proof}

\begin{lemma}\label{eq of stable}
Let $I$ be a monomial ideal of $R=K[x_1,\ldots,x_n]$. 
Then $I$ is stable if and only if 
$(I;x_n,x_{n-1},\ldots,x_1)$ has the complete $\Gm$-full property. 
\end{lemma}

\begin{proof}
The ``if'' part follows from Example 17 in \cite{HW}. 
So we show the ``only if'' part. 
Let $u_1,\ldots,u_s$ be monomials in $I:\Gm$ 
such that $\{\overline{u_1},\ldots,\overline{u_s}\}$ is 
a minimal generating set of $(I:\Gm)/I$, 
where $\overline{u_i}$ is the image of $u_i$ in $R/I$. 
Then it follows from Remark~\ref{notation of mfull} that 
$\{x_nu_1,\ldots,x_nu_s \}$ is a part of a minimal generating set of $I$. 
Write a minimal generating set of $I$ as 
\[ 
\GB=\{ x_nu_1,\ldots,x_nu_s, v_1,\ldots,v_t \}
\] 
where $v_1,\ldots,v_t$ are also monomials of $I$. 
This is the unique minimal set of monomial generators of $I$. 
Hence, 
to verify that $I$ is stable, 
it suffices to show that, for each $w \in \GB$, 
$x_iw/x_{m(w)} \in I$ for every $i<m(w)$. 
Since $u_j\in I:\Gm$, 
it follows that $x_i(x_nu_j)/x_n = x_iu_j \in I$. 
Furthermore it follows from Lemma~\ref{gen of mfull} (1) that 
$\{\overline{v_1},\ldots,\overline{v_t}\}$ is a minimal generating set of $\overline{I}$ in $R/x_nR$, 
and hence $x_n$ does not divide $v_j$ for all $j$. 
Therefore, 
by an inductive argument on the number of variables, 
we have that $x_iv_j/x_{m(v_j)} \in I$ for every $i<m(v_j)$.  
\end{proof}

\begin{lemma}\label{Gin is CMF}
Let $I$ be a completely $\Gm$-full ideal of the polynomial ring  $R=K[x_1,\ldots,x_n]$. 
Then $(\Gin(I); x_n,x_{n-1},\ldots,x_1)$ has the completely $\Gm$-full property. 
\end{lemma}

\begin{proof}
Let $B(I)$ and $B(\Gin(I))$ be 
the sums of the $t$-sequences of $I$ and $\Gin(I)$ respectively 
as in Theorem~\ref{char of cm-full}. 
Since the $t$-sequence of $I$ coincides with that of $\Gin(I)$ by Remark~\ref{t-seq of I and GinI}, 
we see that $B(I)=B(\Gin(I))$. 
Furthermore 
the equality $\mu(I)=B(I)$ holds by Theorem~\ref{char of cm-full} 
and the inequality $\mu(\Gin(I)) \leq B(\Gin(I))$ holds by Lemma~\ref{inequality of t-seq}. 
Hence we have 
\[
B(\Gin(I)) = B(I) = \mu(I) \leq \mu(\Gin(I)) \leq B(\Gin(I)). 
\]
Therefore the equality $\mu(\Gin(I)) = B(\Gin(I))$ holds. 
Thus $\Gin(I)$ is completely $\Gm$-full by Theorem~\ref{char of cm-full}. 
\end{proof}

\begin{lemma}\label{Hilb of full}
Let $I$ be an $\Gm$-full ideal of $R$, 
and assume that $\Gin(I)$ is $\Gm$-full. 
Then 
\[
l((I:\Gm)/I) = l((\Gin(I):\Gm)/\Gin(I)). 
\]
\end{lemma}

\begin{proof} 
It suffices to show that 
\[
\Hilb((I:\Gm)/I,j)=\Hilb((\Gin(I):\Gm)/\Gin(I),j)
\] 
for all $j$. 
Set $J=\Gin(I)$. 
Since $I$ and $J$ are $\Gm$-full, 
there exists a general linear form $z$ of $R$ satisfying $\Gm I:z=I$ and $\Gm J:z=J$. 
Then it is easy to see that $I:\Gm=I:z$ and $J:\Gm=J:z$. 
Hence, from the exact sequence 
\[
0 \rightarrow (I:\Gm)/I \rightarrow R/I \stackrel{\times z}{\rightarrow} 
R/I \rightarrow R/(I+zR) \rightarrow 0, 
\] 
we have 
\[
\Hilb((I:\Gm)/I,j-1)=\Hilb(R/I+zR,j)-\Hilb(R/I,j)+\Hilb(R/I,j-1)
\] 
for all $j$. 
Similarly we have 
\[ 
\Hilb((J:\Gm)/J,j-1)=\Hilb(R/J+zR,j)-\Hilb(R/J,j)+\Hilb(R/J,j-1)  
\] 
for all $j$. 
Recall the well-known facts: 
\begin{itemize}
\item 
$\Hilb(R/I,j)=\Hilb(R/J,j)$ for all $j$. 
\item 
$\Hilb(R/(I+zR),j)=\Hilb(R/(J+zR),j)$ for all $j$ (\cite[Lemma 1.2]{C}). 
\end{itemize}
Hence we get the desired equalities. 
\end{proof}

We are ready to prove Theorem~\ref{main theorem}.

\begin{proof}[Proof of Theorem~\ref{main theorem}.] 
As mentioned in Introduction, 
it suffices to show that 
if $I$ is completely $\Gm$-full, 
then $\Gin(I)$ is stable and $\mu(I)=\mu(\Gin(I))$. 

First note that $\Gin(I)$ is stable. 
This follows from Lemmas~\ref{eq of stable} and~\ref{Gin is CMF}. 
Next we show that $\mu(I)=\mu(\Gin(I))$. 
After a generic linear change of variables 
we may assume that 
$(I; x_n,x_{n-1},\ldots,x_{1})$ has the complete $\Gm$-full property. 
Since $I$ and $\Gin(I)$ are $\Gm$-full, 
it follows by Lemma~\ref{gen of mfull} (2) that 
\[
\mu(I)=\mu(\overline{I})+l((I:\Gm)/I) \ \ 
\mbox{ and } \ \ \mu(J)=\mu(\overline{J})+l((J:\Gm)/J). 
\]
Since $\overline{J}$ is the generic initial ideal of $\overline{I}$ 
(\cite[Corollary 2.15]{Gr}), 
it follows by an inductive argument on the number of variables 
that $\mu(\overline{I}) = \mu(\overline{J})$.  
Hence the equality $\mu(I)=\mu(J)$ holds by Lemma~\ref{Hilb of full}. 
\end{proof}

\begin{remark}
{\rm 
If $K$ is a finite field, an ideal  can be  componentwise linear without being  
completely $\Gm$-full.  To construct an example, suppose that 
$R=K[x_1, \ldots, x_n]$ is the polynomial ring over a finite field $K$. Assume $n \geq 2$.
Let $f$ be the product of {\em all} linear forms in $R$ and let  $I$ be the ideal generated 
by $f$ and $(x_1, \ldots, x_n)^{d+1}$, where $d=\deg f$.  
Then it is easy to see that the ideal $I$ is componentwise linear but not  $\Gm$-full. 
On the other hand if  an ideal  $I \subset R$ is completely $\Gm$-full, 
then $I$ is  necessarily componentwise linear. 
To see this let $K'$ be an infinite field containing $K$ and   $R'=R \otimes _K K'$.  
If  $I \subset R$ is  completely $\Gm$-full, then 
 the ideal $I' = I\otimes _K K'$ is completely $\Gm$-full in $R'$. 
By Theorem~\ref{main theorem} $I'$ is componentwise linear. This implies that $I$ is componentwise linear, since  
a minimal free resolution of $I_{<j>}$ over $R$ for any $j$ induces a minimal free resolution of 
$(I')_{<j>}$ over $R'$. 
}
\end{remark}

\begin{remark}
{\rm 

The original definition of $\mathfrak{m}$-fullness 
was suggested to the second author by Rees himself (see Introduction in \cite{Watanabe}):  
An ideal $\mathfrak{a}$ of  a local ring $(R,\mathfrak{m})$ is called $\mathfrak{m}$-full 
if $\mathfrak{a}\mathfrak{m}:y=\mathfrak{a}$ for some $y$ in a certain faithfully flat extension of $R$.  
If we use this definition, 
Theorem~\ref{main theorem} is true without assuming $K$ to be infinite. 

}
\end{remark}


\section{Componentwise linear ideals of low type} 

In this section we give a generalization of a theorem of Nagel and R\"{o}mer, 
which states that a componentwise linear Gorenstein ideal 
exists only in embedding dimension one.  
The following is a consequence of Theorem~\ref{main theorem}.

\begin{proposition}\label{generalization}
Suppose that $R=K[x_1, \ldots, x_n]$ is the polynomial ring 
and $I$ is a componentwise linear ideal of height $h$ 
such that $R/I$ is Cohen-Macaulay.  
If the type of $I$ is  $r$ and $h \geq r$, 
then $I$ contains $h-r$ linearly independent linear forms. 
\end{proposition}

\begin{proof}
First note that $I$ is completely $\Gm$-full by Theorem~\ref{main theorem}. 
Since $R/I$ is Cohen-Macaulay of dimension $n-h$, 
there exists a regular sequence mod $I$ 
consisting of $n-h$ linear forms in $R$, 
says $\{y_1,\ldots,y_{n-h}\}$. 
Let $\overline{I}$ be the image of $I$ in $\overline{R}=R/(y_1,\ldots,y_{n-h})R$. 
Then $\overline{I}$ is also completely $\Gm$-full in $\overline{R}$ by Corollary~\ref{Cor:reduction}, 
and hence the equality $\overline{\Gm}\overline{I}:\overline{z}=\overline{I}$ 
holds for some linear form $z$ in $R$. 
Therefore it follows that 
$\overline{I}:\overline{\Gm}=\overline{I}:\overline{z}$ (see Remark~\ref{notation of mfull}), 
and  
\[ \begin{array}{rcl}
r &=& l((I:\Gm)/I) = l((\overline{I}:\overline{\Gm})/\overline{I}) 
= l((\overline{I}:\overline{z})/\overline{I}) \\
 &=& l(\overline{R}/(\overline{I}+\overline{z}\overline{R}))
= l(R/(I+(z,y_1,\ldots,y_{n-h})R)).
\end{array}
\] 
Since $l(R/(I+(z,y_1,\ldots,y_{n-h})R)) \leq h$ by assumption, 
it follows that $I$ must contain 
a regular sequence consisting of $h-r$ linearly independent linear forms. 
Those linear forms are members of a minimal generating set of $I$.  
\end{proof}

\begin{corollary}[Nagel-R\"omer]\label{level one} 
Suppose that $R=K[x_1, \ldots, x_n]$ is the polynomial ring 
and $I$ is a Gorenstein ideal of height $h$. 
Then $I$ is componentwise linear if and only if 
$I$ is a complete intersection ideal minimally generated by at least $h-1$ linear forms. 
\end{corollary}

\begin{proof}
The ``if" part follows from Proposition~\ref{generalization}. 
The ``only if" part: 
By Corollary~\ref{Cor:reduction} 
it suffices to prove it in the case where $I$ is $\Gm$-primary. 
It is obvious that $I$ is Gorenstein. 
By assumption, 
there exist $n$ linear forms $z_1,\ldots,z_{n-1},z_n$ and an integer $d>0$ 
such that $I=(z_1,\ldots,z_{n-1},z_n^d)$. 
Let $\overline{I}$ be the image of $I$ in $R/z_nR$. 
Then the equality $\mu(I)-\mu(\overline{I})=l(R/I+z_nR)$ holds 
because $\mu(I)=n$, $\mu(\overline{I})=n-1$ and $l(R/I+z_nR)=1$. 
Hence $I$ is $\Gm$-full by Lemma~\ref{eq of m-full}. 
Furthermore it is obvious that 
$\overline{I}=(\overline{z_1},\ldots,\overline{z_{n-1}})$ 
is completely $\Gm$-full in $R/z_nR$. 
Therefore $I$ is completely $\Gm$-full, 
and hence $I$ is compenetwise linear by Theorem~\ref{main theorem}. 
\end{proof}

This was proved by Nagel and R\"{o}mer in Theorem 3.1 of \cite{NR}. 
Our proof as a corollary of 
Theorem~\ref{main theorem} and Proposition~\ref{generalization} 
is completely different from theirs.  
There are also similar results 
in \cite[Theorem 1.1]{G} and \cite[Proposition 2.4]{GH}.



%

\end{document}